\documentclass[reqno,11pt]{amsart}
\usepackage{graphicx,amsfonts,amssymb,amsmath,amsthm,amscd}
\usepackage[all]{xy}
\usepackage{color}
\usepackage{hyperref}
\hypersetup{colorlinks,linkcolor=magenta,citecolor=blue,filecolor=blue,urlcolor=blue}
\newcounter{statement}

\newtheorem{theorem}    [statement]{Theorem}
\newtheorem{lemma}      [statement]{Lemma}

\newtheorem{prop}[statement]{Proposition}

\newcounter{intro}

\newtheorem*{theorem*}           {Theorem}
\newtheorem*{conjecture*}           {Conjecture}

\theoremstyle{definition}
\newtheorem{definition} [statement]{Definition}

\newtheorem*{definition*}           {Definition}

\theoremstyle{remark}
\newtheorem{remark}[statement]{Remark}

\newtheorem*{claim*}               {Claim}

\definecolor{prpl}{rgb}{0.7, 0.0, 0.7}
\newenvironment{JT}{\noindent \color{prpl}{\bf JT:} \footnotesize}{} \newenvironment{BB}{\noindent \color{blue}{\bf BB:} \footnotesize}{}

\def\chone{\operatorname{c_1}}

\def\rk{\operatorname{rk}}

\def\Pic{\operatorname{Pic}}

\def\ev{\operatorname{ev}}
\def\Cliff{\operatorname{Cliff}}

\def\Z{\mathbb{Z}}

\def\cM{\mathcal{M}}

\def\mm{\overline{\mathcal{M}}}

\def\cM{\mathcal{M}}

\def\cZ{\mathcal{Z}}
\def\cU{\mathcal{U}}

\def\pp{{\textbf P}}

\renewcommand{\O}{\mathcal{O}}

\renewcommand{\phi}{\varphi}

\def\into{\rightarrow}

\title{The Mercat Conjecture for stable rank $2$ vector bundles on generic curves}
\author[B. Bakker]{Benjamin Bakker}
\address{Humboldt-Universit\"at zu Berlin, Institut f\"ur Mathematik,  Unter den Linden 6
\hfill \newline\texttt{}
\indent 10099 Berlin, Germany} \email{{\tt benjamin.bakker@math.hu-berlin.de}}

\author[G. Farkas]{Gavril Farkas}
\address{Humboldt-Universit\"at zu Berlin, Institut f\"ur Mathematik,  Unter den Linden 6
\hfill \newline\texttt{}
 \indent 10099 Berlin, Germany} \email{{\tt farkas@math.hu-berlin.de}}

\begin{document}

\begin{abstract}
We prove the Mercat Conjecture for rank $2$ vector bundles on generic curves of every genus. For odd genus, we identify the effective divisor in $\cM_g$ where the Mercat Conjecture fails and compute its slope.
\end{abstract}

\maketitle

The Clifford index of a smooth curve $C$, defined as the non-negative quantity
$$\mbox{Cliff}(C):=\mbox{min}\Bigl\{\mbox{deg}(L)-2r(L): L\in \mbox{Pic}(C),\ h^0(C,L)\geq 2, \ \mbox{deg}(L)\leq g-1\Bigr\}$$
stratifies the moduli space $\cM_g$ of curves of genus $g$, with the smallest stratum being that consisting of hyperelliptic curves, which have Clifford index zero. While the definition is obviously inspired by the classical Clifford theorem, the notion came to real prominence in the works of Green, Lazarsfeld \cite{G}, \cite{GL}, Voisin \cite{V} and many others in the context of syzygies of canonical curves. For a general curve $C$ of genus $g$, one has $\mbox{Cliff}(C)=\lfloor \frac{g-1}{2}\rfloor$; the locus of curves $[C]\in \cM_g$ with $\mbox{Cliff}(C)<\lfloor \frac{g-1}{2}\rfloor$  is a subvariety of codimension $1$ (respectively $2$) when $g$ is odd (respectively even) much studied by Harris and Mumford \cite{HM}.

\vskip 3pt

It has been a long standing problem to find an adequate definition of the Clifford index for higher rank vector bundles on curves. A higher rank Clifford index should not only capture the behavior of the generic curve from the point of view of special higher rank vector bundles, but also provide a geometrically meaningful stratification of $\cM_g$. An interesting  definition has been put forward by Lange and Newstead \cite{LN1}. For a semistable vector bundle $E$ of rank $2$ and slope $\mu(E)$  on a curve $C$ of genus $g\geq 4$, one defines its Clifford index as
$$\mathrm{Cliff}(E):=\mu(E)-h^0(C, E)+2\ge 0.$$ The \emph{rank $2$ Clifford index} of $C$ is then defined as the quantity
$$\mathrm{Cliff}_2(C):=\mbox{min}\Bigl\{\mathrm{Cliff}(E): E\in \cU_C(2, d), \ \ d\leq 2g-2, \ \ h^0(C, E)\geq 4\Bigr\}.$$

Observe that if $L$ is a line bundle on $C$ with $\mbox{deg}(L)-2h^0(C,L)+2=\mbox{Cliff}(C)$ (that is, $L$ computes the Clifford index of $C$), then $\mbox{Cliff}(L\oplus L)=\mbox{Cliff}(C)$. In particular, the inequality $\mbox{Cliff}_2(C)\leq \mbox{Cliff}(C)$ holds for every curve $C$ of genus $g$. The main achievement of this paper is the proof of the following result:

\begin{theorem}\label{mercat2}
For a general curve $C$ of genus $g$ the following equality holds
$$\mathrm{Cliff}_2(C)=\mathrm{Cliff}(C)=\Bigl\lfloor \frac{g-1}{2}\Bigr\rfloor.$$
Equivalently, if $E$ is a semistable rank $2$ vector bundle on $C$ contributing to $\mathrm{Cliff}_2(C)$, then
$$h^0(C,E)\leq \mu(E)+2-\mathrm{Cliff}(C).$$
\end{theorem}

Mercat \cite{Me} conjectured the equality $\mbox{Cliff}_2(C)=\mbox{Cliff}(C)$ for every smooth curve $[C]\in \cM_g$. Counterexamples to this expectation have been found using Noether-Lefschetz special $K3$ surfaces in \cite{FO1}, \cite{FO2}, \cite{LN2} and \cite{LN3}. However these examples are special in moduli and Theorem \ref{mercat2} proves Mercat's Conjecture for generic curves of every genus.

\vskip 3pt

In view of Theorem \ref{mercat2}, it is interesting to describe the cycle in $\cM_g$ consisting of curves $C$ such that $\mathrm{Cliff}_2(C)<\mathrm{Cliff}(C)$. Recall that for odd genus, the locus of curves having sub-maximal Clifford index can be identified with the \emph{Hurwitz divisor}
$$\mathfrak{Hur}_g:=\Bigl\{[C]\in \cM_g: C \mbox{ carries a pencil } \mathfrak g^1_{\frac{g+1}{2}}\Bigr\}.$$
The divisor $\mathfrak{Hur}_g$ was instrumental in the proof of Harris-Mumford \cite{HM} that $\cM_g$ is of general type for large $g$. It has already been pointed out in \cite{FO1}, \cite{FO2}, \cite{LN2} that there exist curves $C$ with maximal Clifford index $\mathrm{Cliff}(C)=\bigl \lfloor \frac{g-1}{2} \bigr \rfloor$, but with $\mbox{Cliff}_2(C)<\mbox{Cliff}(C)$. In particular, the stratification of $\cM_g$ given by the rank $2$ Clifford index is \emph{different} from the classical stratification given by gonality.  We show that, for odd genus, the locus of curves with a sub-maximal second Clifford index is a \emph{Koszul divisor} on $\cM_g$, different from the Hurwitz divisor but having the same slope!

For a globally generated line bundle $L$ on a curve $C$, the Koszul cohomology group
$$K_{1,1}(C,L):=\mbox{Ker}\Bigl\{\mbox{Sym}^2 H^0(C,L)\rightarrow H^0(C, L^{\otimes 2})\Bigr\}$$ is identified with the space of quadrics containing the image curve $\phi_L:C\rightarrow \pp\bigl(H^0(C,L)^{\vee}\bigr)$.

\begin{theorem}\label{divisor}
For each odd genus $g\geq 11$, the locus of curves $C$ satisfying $\mathrm{Cliff}_2(C)<\mathrm{Cliff}(C)$ has a divisorial component which is identified with the Koszul divisor
$$\mathfrak{Kosz}_g:=\Bigl\{[C]\in \cM_g:\exists L\in W^4_{g+2}(C)  \  such \  that  \ \ K_{1,1}(C,L)\neq 0\Bigr\}.$$
The locus $\mathfrak{Kosz}_g$ is an effective divisor on $\cM_g$, different from the Hurwitz divisor, of slope
$$s(\overline{\mathfrak{Kosz}}_g)=6+\frac{12}{g+1}.$$
\end{theorem}

Recall that if $D$ is an effective divisor on $\mm_g$ not containing any boundary divisors in its support and with $[D]=a\lambda-\sum_{i=0}^{\lfloor \frac{g}{2}\rfloor} b_i\delta_i\in \mbox{Pic}(\mm_g)$, its slope, defined by $s(D):=\frac{a}{b_0}$, is a quantity measuring the positivity of $D$. Harris and Mumford \cite{HM} famously showed that $s(\overline{\mathfrak{Hur}}_g)=6+\frac{12}{g+1}$, which implies that $\mm_g$ is of general type for odd genus $g>23$ . Our Theorem \ref{divisor} shows  that  the divisors $\overline{\mathfrak{Kosz}}_g$ and $\overline{\mathfrak{Hur}}_g$ have the same slope.

\vskip 4pt

A parameter count indicates that $\mathfrak{Kosz}_g$ is indeed expected to be a divisor on $\cM_g$. For a general curve $C$ of genus $g\geq 11$ and a base point free linear series $L\in W^4_{g+2}(C)$, since $H^1(C,L^{\otimes 2})=0$  we have that
$$h^0(C,L^{\otimes 2})-\mbox{dim } \mbox{Sym}^2 H^0(C,L)+1=\mbox{dim } W^4_{g+2}(C)+1=g-9.$$
Furthermore, the condition $K_{1,1}(C,L)\neq 0$ is equivalent to the existence of a globally generated rank $2$ vector bundle $E$ on $C$ with $\mbox{det}(E)=L$ and $h^0(C,E)=4$. It is straightforward to show that for a general point $[C]\in \mathfrak{Kosz}_g$, the corresponding vector bundle $E$ is stable. Clearly
$\mbox{Cliff}(E)=\frac{g-2}{2}<\bigl \lfloor \frac{g-1}{2}\bigr \rfloor$, that is, $\mbox{Cliff}_2(C)<\bigl \lfloor \frac{g-1}{2}\bigr \rfloor $ for all points $[C]\in \mathfrak{Kosz}_g$.
Syzygy divisors on $\overline{\cM}_g$ characterized by the non-vanishing of a Koszul cohomology groups $K_{p,1}(C,L)$ for a line bundle $L$ with Brill-Noether number equal to zero have been studied systematically in \cite{Fa} and shown to lead to divisors of slope less that $6+\frac{12}{g+1}$. In light of those results, the fact that $\overline{\mathfrak{Kosz}}_g$ has the same slope as that of the Brill-Noether divisors (despite being distinct from them) comes as a surprise.

\vskip 3pt
The proof of both Theorems \ref{mercat2} and \ref{divisor} relies in an essential way on a specialization to generic polarized $K3$ surfaces. Precisely we establish the following result:

\begin{theorem} \label{k3mercat}  Let $(X,H)$ be a smooth polarized $K3$ surface of degree $H^2=2g-2$ such that $\mathrm{Pic}(X)=\mathbb Z\cdot H$. Then for any (possibly nodal) curve $C\in |H|$, we have that
$$\Cliff_2(C)=\Cliff(C).$$
\end{theorem}

The definition of $\mbox{Cliff}_2(C)$ can be generalized in a straightforward manner to irreducible curves and torsion-free sheaves on them (see Section 1). The proof of Theorem \ref{k3mercat} uses \emph{Lazarsfeld-Mukai bundles} on $K3$ surfaces. Let $C\in |H|$ be a curve having at worst nodal singularities and $E$ a semistable rank $2$ vector bundle on $C$ which contributes to the second Clifford index $\mathrm{Cliff}_2(C)$. We may assume that $E$ is globally generated, then consider the vector bundle $F:=F_E$ obtained by an elementary transformation from the trivial vector bundle on $X$ using the sections of $E$:
$$0\longrightarrow F_E\longrightarrow H^0(C,E)\otimes \mathcal{O}_X\stackrel{\mathrm{ev}}\longrightarrow E\longrightarrow 0.$$
A crucial observation is that the Lazarsfeld-Mukai bundle $F_E$ is  Gieseker unstable whenever the inequality $\mbox{Cliff}(E)\leq \bigl \lfloor \frac{g-1}{2}\bigr \rfloor$ holds. Considering the not so numerous possibilities for the shape of the Harder-Narasimhan filtration of $F_E$, we quickly reach a contradiction with the assumption $\mbox{Pic}(X)=\mathbb Z\cdot H$. Thus no vector bundle $E$ on $C$ satisfying $\mbox{Cliff}(E)< \bigl \lfloor \frac{g-1}{2} \bigr \rfloor$ can exist in the first place, which establishes Theorem \ref{k3mercat}. The fact that the value of $\mbox{Cliff}(E)$ is tightly linked to the stability of the associated Lazarsfeld-Mukai bundle $F_E$ offers another strong indication that the definition of $\mbox{Cliff}_2(C)$ is the correct one to capture the geometry of rank $2$ vector bundles on generic curves.

\vskip 3pt

Our study of rank $2$ vector bundles on curves via Lazarsfeld-Mukai (L-M) bundles is reminiscent of Lazarsfeld's proof of the Brill-Noether Theorem \cite{La}, where crucially, the L-M vector bundle $F_A$ associated to a globally generated line bundle $A\in W^r_d(C)$ is non-simple precisely when the Brill-Noether number $\rho(g,r,d)=g-(r+1)(g-d+r)<0$. This implies that no linear series $A\in W^r_d(C)$ with $\rho(g,r,d)<0$ can exist on the curve  $C\in |H|$. A link between our vector bundle $F_E$ and the rank one situation discussed in \cite{La} can be obtained  via a degeneration of the Mukai system to the Hitchin system, see also \cite{DEL}. Letting the $K3$ surface $X$ degenerate to the cone over a smooth curve $C\in |H|$, the linear system $|2H|$ on $X$ specializes to the system of degree $2$ spectral covers of $C$. The vector bundle $F_E$ in turn specializes to a classical
L-M bundle $F_{A'}$ on a curve $C'\in |2H|$, where $A'\in \mbox{Pic}^{2\mu(E)+2g-2}(C')$. However, unlike the situation in \cite{La}, the curve $C'$ is far from being Brill-Noether general. Still, the corresponding vector bundle $F_{A'}$ displays enough regularity properties to prove Theorem \ref{k3mercat}.

\vskip 4pt
The proof of Theorem \ref{divisor} follows directly from Theorem \ref{k3mercat}, by observing that the rational curve induced by a Lefschetz pencil of curves in the linear system $|H|$ on a fixed general $K3$ surface $X$ of degree $2g-2$ is disjoint from the divisor $\overline{\mathfrak{Kosz}}_g$. In particular, the slope of the Koszul divisor is equal to the ratio $j^*(\delta_0)/j^*(\lambda)=6+\frac{12}{g+1}.$

\vskip 2pt

We close the introduction by mentioning another application that follows from Theorem \ref{k3mercat}. The classical Clifford index is known to be intimately related to the gonality of the curve, in the sense that for a general curve $C$, the Clifford index is computed by a line bundle $L\in \mbox{Pic}(C)$ with $h^0(C,L)=2$, and accordingly $\mbox{Cliff}(C)=\mbox{gon}(C)-2$, see \cite{Ba}. A similar result holds for rank $2$ vector bundles as well, provided the genus is not too small:

\begin{theorem}\label{4sect}
If $C$ is a general curve of even genus $g\geq 10$ and $E$ is a rank $2$ vector bundle on $C$ computing $\mathrm{Cliff}_2(C)$, then $h^0(C,E)=4$.
\end{theorem}

We conjecture a similar result for a general curve $C$ of odd genus $g\geq 15$ as well. In Section 3 we show this is equivalent to a plausible maximal rank statement for a multiplication map of global sections of line bundles on $C$.

\section{Vector bundles on curves and the Mercat Conjecture}
Let $C$ be a smooth curve of genus $g\geq 4$ and $E$ be a semistable rank $r$ vector bundle on $C$ with $h^1(C,E)\geq h^0(C,E)\geq 2r$.  Recall that the Clifford index of $E$ is defined as the quantity
 \begin{equation}\Cliff(E):=\mu(E)-\frac{2}{r} h^0(C,E)+2=(g+1)-\frac{h^0(C,E)+h^1(C,E)}{r}\geq 0. \label{cliffdef}\end{equation}
Whenever we take the Clifford index of a vector bundle, we implicitly assume it contributes to the Clifford index, in the sense that it satisfies the above conditions.   Classically, Clifford's theorem asserts that for a line bundle $L$ contributing to the Clifford index, one has $\Cliff(L)\geq 0$.

\vskip 3pt

The $r$th Clifford index $\Cliff_r(C)$ of $C$ is defined as the minimum of $\Cliff(E)$ considered over all \emph{semistable} rank $r$ vector bundles $E$ contributing to the Clifford index, see \cite{LN1}.  We write $\Cliff(C)=\Cliff_1(C)$ for the classical Clifford index.
Taking $E$ to be a direct sum of line bundles of the same degree implies the obvious inequality $\Cliff_r(C)\leq \Cliff(C)$.

\vskip 3pt

Mercat \cite{Me} conjectured that for every smooth curve $C$, the equality
$$(M_r): \ \ \Cliff_r(C)=\Cliff(C).$$
should hold.  The conjecture $(M_2)$ amounts thus to the inequality
\begin{equation}\label{mercatin}
h^0(C,E)\leq 2+\mu(E)-\mathrm{Cliff}(C),
\end{equation}
for every semistable rank $2$ bundle $E$ with $\mu(E)\leq g-1$ and $h^0(C,E)\geq 4$. As already pointed out, inequality (\ref{mercatin}) is sharp for every curve $C$, for there exist strictly semistable bundles $E$ for which (\ref{mercatin}) is an equality. Conjecture $(M_2)$  holds for every curve $C$ of genus $g\leq 10$, see \cite{LN1} and was known to be false for curves of genus $g\geq 11$ lying on special $K3$ surfaces, see \cite{FO1},\cite{FO2}, \cite{LN2}. The Mercat Conjecture remained plausible
in the most interesting case, that of a general curve $C$, in which case the Brill--Noether theorem implies that $\Cliff(C)=\bigl \lfloor\frac{g-1}{2}\bigr \rfloor$. The main result of this paper is a proof of the generic conjecture $(M_2)$.

\vskip 4pt

For higher rank the situation is less clear and the evidence for the generic conjecture $(M_r)$ is somewhat weaker. Despite some positive results in \cite{LN1}, it has been shown in \cite{LMN} that conjecture $(M_3)$ fails for a \emph{general} curve of genus $9$ or $11$, while conjecture $(M_5)$ fails for a general curve of genus $7$, see \cite{AFO}. In each of these cases, the stable vector bundle $E$ on a general curve $C$ satisfying the inequality $\mbox{Cliff}(E)<\mbox{Cliff}(C)$ is (a twist of) the \emph{Mukai vector bundle} on $C$  introduced in \cite{M2} and  \cite{M3} in order to give structure theorems for the moduli space of curves of genus $g=7,8,9$. As pointed out in \cite{FO2} Theorem 1.5, the Mercat Conjectures $(M_3)$ and $(M_4)$ fail for any curve $C\in |H|$ of high enough genus lying on a $K3$ surface $X$ with $\mbox{Pic}(X)=\mathbb Z\cdot H$. Therefore the analogue of Theorem \ref{k3mercat} in higher rank is false.

\vskip 5pt

\subsection{Rank $2$ bundles with $4$ sections vs. rank $6$ quadrics} We often use a well-known equivalence between the following set of pairs of objects on a curve $C$, see
\cite{V}, but also \cite{GMN} or \cite{FO1},
\begin{equation}\label{equiv1}
\Bigl\{(E,V):V\subset H^0(C,E), \ \mbox{dim}(V)=4\Bigr\}\stackrel{1:1}\longleftrightarrow \Bigl\{(L,q):0\neq q\in K_{1,1}(C,L), \ \mbox{rk}(q)\leq 6\Bigr\},
\end{equation}
where, on one side, $E$ is a rank $2$ vector bundle on $C$ together with a space of sections $V$ generating $E$ and inducing a morphism $\varphi_V:C\rightarrow G(2,V^{\vee})$ to the Grassmannian $G(2,4)$, which is identified with the Pl\"ucker quadric. On the other side, $L\in \mbox{Pic}(C)$ is a globally generated line bundle and $q$ is a quadric hypersurface of rank at most $6$ containing the image curve $\varphi_L: C\rightarrow \pp \bigl(H^0(C,L)^{\vee}\bigr)$. Starting with the pair $(E,V)$, we set $L:=\mbox{det}(E)$ and denote by $\lambda: \bigwedge ^2 V \rightarrow H^0(C,L)$ the determinant map. We consider the following composition
$$C\stackrel{\varphi_V}\longrightarrow G\bigl(2, V^{\vee}\bigr) \longrightarrow \pp \Bigl(\bigwedge^2 V^{\vee}\Bigr) \stackrel{\pp(\lambda^{\vee})}\longleftarrow \pp \Bigl(H^0(C,L)^{\vee}\Bigr),$$
and set $q:=\varphi_L^{-1}\bigl(G(2,V^{\vee})\bigr)\in K_{1,1}(C,L)$. Concretely, if $(s_1,s_2,s_3,s_4)$ is a basis of $V$, then
$$q=\lambda(s_1\wedge s_2)\cdot \lambda(s_3\wedge s_4)-\lambda(s_1\wedge s_3)\cdot \lambda(s_2\wedge s_4)+\lambda(s_1\wedge s_4)\cdot \lambda(s_2\wedge s_3).$$
Observe that if $E$ contains no subpencils, (that is, line bundles $A\hookrightarrow E$ with $h^0(C,A)\geq 2$), then $\mbox{rk}(q)\in \{5,6\}$. Else, $\mbox{rk}(q)\leq 4$.

\vskip 3pt

As pointed out in \cite{Me}, the inequality (\ref{mercatin}) is valid for any semistable rank $2$ vector bundle $E$ having a subpencil. When $E$ has no subpencils, then $h^0(C,L)\geq 2h^0(C, E)-3$. In particular, if $E$ contributes to $\mbox{Cliff}_2(C)$ then $h^0(C, \mbox{det}(E))\geq 5$. This inequality and the equivalence (\ref{equiv1}) explain the peculiar form of the divisor $\mathfrak{Kosz}_g$ in Theorem \ref{divisor}.



\section{Moduli of sheaves on $K3$ surfaces}

Let $X$ be a smooth $K3$ surface with $\Pic(X)=\Z\cdot H$, where $H$ is an ample generator of degree $H^2=2g-2$.  We set $H^{\bullet}(X):=H^0(X,\mathbb Z)\oplus H^2(X, \mathbb Z)\oplus H^4(X, \mathbb Z)$. Following \cite{M1} or \cite{HL}, we define the \emph{Mukai pairing} on $H^{\bullet}(X)$ by
$$(v_0, v_1, v_2)\cdot (w_0. w_1, w_2):=v_1\cdot w_1-v_2\cdot w_0-v_0\cdot w_2\in H^4(X, \mathbb Z)\cong \mathbb Z.$$
For a sheaf $F$ on $X$, its \emph{Mukai vector} is defined following  \cite{M1} Definition 2.1, by setting
\[v(F):=\Bigl(\rk(F),\chone(F),\chi(F)-\rk(F)\Bigr).\]
Note that we have $-\chi(F,F)=v(F)^2$.   Throughout the paper, by (semi-)stability we mean Gieseker (semi-)stability with respect to the polarization $H$. We denote by $M_H(v)$ the moduli space of $S$-equivalence classes of $H$-semistable sheaves $F$ on $X$ having prescribed Mukai vector $v(F)=v$. Let $M_H^s(v)$ be the open subset of $M_H(v)$ corresponding to $H$-stable sheaves. Then it is known that $M_H^s(v)$ is pure dimensional and $\mbox{dim } M_H^s(v)=v^2+2$. In particular, if $v(F)^2<-2$ for a sheaf $F$ on $X$, then $F$ is not stable.

\vskip 3pt
For an irreducible possibly singular curve $C$ with surface singularities, we can still define the Clifford index of a (Gieseker) semistable torsion-free sheaf $E$ by formula \eqref{cliffdef}, where $g:=p_a(C)$ is the arithmetic genus. One can likewise define $\Cliff_r(C)$.  It is a simple observation that Lazarsfeld's result \cite{La} on the Brill--Noether generality of a hyperplane section of a general $K3$ surface extends to singular curves, which we include for completeness:
\begin{prop}\label{k3clifford} Let $X$ be a K3 surface with $\Pic(X)=\Z\cdot H$ and ample generator $H$ of degree $H^2=2g-2$.  For any section $C\in|H|$, one has $\Cliff(C)=\lfloor\frac{g-1}{2}\rfloor$, and any torsion-free sheaf $L$ computing the Clifford index has $h^0(C,L)=2$.
\end{prop}

\begin{proof} The statement is true for a smooth section $C\in|H|$, see \cite{La}.  The relative compactified Jacobian of the universal divisor in the linear system $|H|$ parametrizes torsion-free sheaves supported on an element of $|H|$ and with rank $1$ on their support.  Thus, for any section $C\in|H|$, by semicontinuity, we certainly have $\Cliff(C)\leq \lfloor\frac{g-1}{2}\rfloor$.  Given a torsion-free sheaf $L$ on $C$ having minimal Clifford index, $L$ is necessarily globally generated.  Consider the exact sequence
\[0\longrightarrow F\longrightarrow \O_X\otimes H^0(C,L)\longrightarrow L\longrightarrow 0. \]
The vector bundle $F$ is automatically stable and $v(F)=\Bigl(h^0(C,L),-H,h^1(C,L)\Bigr)$. Furthermore, the inequality $v(F)^2\geq -2$ implies that $h^0(C,L)\cdot h^1(C,L)\leq g$, from which it follows that $\Cliff(L)\geq \lfloor\frac{g-1}{2}\rfloor$, as $h^1(C,L)\geq h^0(C,L)\geq 2$.  If we have equality, then necessarily $h^0(C,L)=2$.
\end{proof}

\vskip 4pt

Now consider a pure $1$-dimensional sheaf $E$ on the $K3$ surface $X$  supported on an irreducible curve $C\in |H|$. Let $r$ be its generic rank on its support.  Note the following simple fact:
\begin{lemma}\label{bigrank} With $E$ as above, suppose we have torsion-free sheaves $M,N$ on $X$ and an exact sequence
\[0\longrightarrow M\longrightarrow  N\longrightarrow E\longrightarrow 0.\]
Then $\rk(M)=\rk(N)\geq r$ with equality if and only if $M\cong N\otimes \mathcal{O}_X(-H)$ in codimension 1.
\end{lemma}
\begin{proof}  Let $\rk(M)=\rk(N)=k$.  Note that we have $\deg(N)=\deg(M)+r$. The sheaf $E$ is scheme-theoretically supported on $C$, so multiplication by the section defining $C$ gives a map $N\otimes \mathcal{O}_X(-H)\into N$ which must factor through $M$, so
\[\deg(N)-k\leq \deg(M)=\deg(N)-r,\]
hence the first claim.  If we have equality, then the injection $N\otimes \mathcal{O}_X(-H)\into M$ is an isomorphism in codimension $1$.
\end{proof}

\vskip 4pt

Assume now that $E$ is a globally generated torsion-free sheaf of rank $r$ on a curve $C\in |H|$ contributing to the Clifford index $\mbox{Cliff}_r(C)$, therefore satisfying
$$\mu(E)\leq g-1 \ \ \mbox{ and } \  \ h^0(C,E)\geq 2r.$$
\begin{definition}
The \emph{Lazarsfeld--Mukai vector bundle} \footnote{Usually the Lazarsfeld--Mukai bundle refers to $F^\vee$, see \cite{GL}, \cite{La} or \cite{AFO}.} $F:=F_{E}$ is defined as the kernel of the evaluation map given by the global sections of $E$:
\begin{equation}\label{lm}
0\longrightarrow  F_E\longrightarrow \O_X\otimes H^0(C,E)\stackrel{\ev}\longrightarrow E\longrightarrow 0.
\end{equation}
\end{definition}

Lazarsfeld-Mukai vector bundles induced by line bundles on curves $C\in |H|$ have been used to great effect effect to prove the Brill-Noether-Petri Theorem \cite{La}, or the generic Green's Conjecture \cite{V}. However, the definition (\ref{lm}) makes sense for a vector bundle $E$ of arbitrary rank on $C$.

Letting $h^1(C,E)=m\geq h^0(C,E)=n$, we compute $v(F)=(n,-rH,m)$ and clearly $H^0(X,F)=H^1(X,F)=0$. There exists a generically surjective morphism $\mathcal{O}_X^{\oplus \mathrm{rk}(F)}\rightarrow F^{\vee}$. Furthermore, one computes
 \begin{align}
 v(F)^2&=-2nm+r^2(2g-2)\notag\\
 &=-2(n-2r)(m-2r)+4r^2\Bigl(\Cliff(E)-\frac{g-1}{2}\Bigr)\label{equality}
 \end{align}
If we assume $\Cliff(E)<\lfloor\frac{g-1}{2}\rfloor$, then since $\Cliff(E)\in \frac{1}{r}\Z$, we have
 \begin{equation}v(F)^2\leq \begin{cases}
 -4r,& \ \mbox{ if } g\mbox{ is odd}\\
 -2r(r+1),& \ \mbox{ if } g \mbox{ is even and } r \mbox{ is odd}\\
 -2(r+2),& \ \mbox{ if } g \mbox{ and } r \mbox{ are both even}
 \end{cases}\label{bogo}\end{equation}
 which implies the following:

\begin{prop}\label{stable} Assume $E$ is a globally generated torsion-free sheaf of rank $r>1$ on a curve $C\in |H|$ contributing to $\mathrm{Cliff}_r(C)$ and such that $\Cliff(E)<\lfloor\frac{g-1}{2}\rfloor$.  Then the associated Lazarsfeld-Mukai vector bundle $F_E$ is not stable.  If $r=2$, then $F_E$ is not even semistable.
\end{prop}
\begin{proof}  The first statement follows from \eqref{bogo} since every stable sheaf is simple.  For the second statement, $v(F)$ is at most 2-divisible so, if $F$ is strictly semistable we must have $$\Bigl(\frac{v(F)}{2}\Bigr)^2=-2.$$  By \eqref{bogo}, this is clearly impossible if $g$ is even; for $g$ odd, $v(F)^2=-8$ implies both $4g=nm$ and $n=4$, so $v(F)$ cannot be divisible.\end{proof}

We can also extract information about $E$ with $\Cliff(E)=\lfloor\frac{g-1}{2}\rfloor$ from the same computation:

\begin{prop}\label{equalcliff}Let $E$ be a globally generated torsion-free sheaf of rank $2$ on $C\in |H|$ contributing to $\mathrm{Cliff}_2(C)$ and such that $\Cliff(E)=\lfloor\frac{g-1}{2}\rfloor$.  Further assume $h^0(C,E)>4$.  Then $F_E$ is not stable unless $g=7$, and not semistable unless $g=7$ or $9$.
\end{prop}
\begin{proof}  For $g$ even the claim is immediate from \eqref{equality} since $v(F)^2\leq -10$.  If $g$ is odd, then $v(F)^2\leq -2(n-4)(m-4)$ and $m+n=g+3$ imply the rest of the claim after a case-by-case analysis.
\end{proof}


Assuming the Lazarsfeld--Mukai bundle $F$ is not semistable, let $G\subset F$ denote the maximal destabilizing sheaf of highest rank, which is necessarily saturated and semistable.  Then we have the following diagram, where we recall that $n=\mbox{rk}(F)$:

\begin{equation}\label{diagram}
\begin{gathered}\xymatrix{
&0\ar[d]&0\ar[d]&&\\
&G\ar@{=}[r]\ar[d]&G\ar[d]&&\\
0\ar[r]&F\ar[r]\ar[d]&\O_X^{\oplus n}\ar[r]^{\ev}\ar[d]&E\ar@{=}[d]\ar[r]&0\\
0\ar[r]&M\ar[r]\ar[d]&N\ar[r]\ar[d]&E\ar[r]&0\\
&0&0&&
}\end{gathered}
\end{equation}

The sheaf $G$, on the one hand, has slope $\mu(G)> \mu(F)$, on the other hand being a subsheaf of $\O_X^{\oplus n}$ has slope $\mu(G)\leq 0$.  The sheaf $M$ is torsion-free, so $G$ is in fact locally free, and therefore cannot be of degree 0, for a generic projection $\O_X^{\oplus n}\into \O_X^{\oplus \rk(G)}$ would give $G\cong \O_X^{\oplus \rk(G)}$ (indeed the determinant of $G\into \O_X^{\oplus \rk(G)}$ is nowhere vanishing), whereas $H^0(X,G)=0$.  Thus, $\mu(G)<0$, so all of the Harder--Narasimhan sub-quotients of $F$ have negative slope, and therefore $H^0(X,M)=0$, whence $H^1(X,G)=0$, and $G$ once again looks like a Lazarsfeld--Mukai bundle.

\vskip 4pt

We can now make some preliminary steps towards the proof of Theorem \ref{k3mercat}.  First, note that a rank 2 (torsion-free) semistable sheaf $E$ on $C$ computing the Clifford index $\Cliff(E)=\Cliff_2(C)\leq \Cliff(C)$ is necessarily globally generated.  Indeed, if not, we have an exact sequence
\[0\longrightarrow  E'\longrightarrow E\longrightarrow P\longrightarrow  0, \]
where $E'$ is the subsheaf of $E$ generated by global sections.  It is easy to see by the assumption on the Clifford index that $E'$ has rank 2, and $\Cliff(E')< \Cliff(E)$.  $E'$ must not be semistable, so there is a Harder--Narasimhan filtration
\[0\into A\into E'\into A'\into 0.\]
On the one hand, $E$ cannot contain a subpencil, but on the other hand by the semistability of $E$ we must have
\[\mu(E)\geq \deg(A)\geq \mu(E')\geq \deg(A').\]
so $A'$ contributes to the rank 1 Clifford index.  We then have $\Cliff(A')<\Cliff(E')$, which is a contradiction.

Whenever the Lazarsfeld--Mukai bundle $F=F_E$ is not semistable, the first column of \eqref{diagram} is the Harder--Narasimhan filtration of $F$ and each of $G$ and $M$ is of degree $-1$, that is, $c_1(G)=c_1(M)=-H$. Indeed, since $c_1(G)=-H$ and $c_1(F)=-2H$, because of the assumption $\mbox{Pic}(X)=\mathbb Z\cdot H$, there is no room for further bundles in the Harder--Narasimhan filtration of $F$, that is, both $G$ and $M$ are stable.   The same is true if we assume $F$ is strictly semistable and let $G\subset F$ be a maximal stable subsheaf in \eqref{diagram}.  For concreteness, set $$v(G)=(k,-H,\ell),$$ so that $v(N)=(n-k,H,n-\ell)$.  Note that since $\mu(G)\geq \mu(F)$, we must have $n>k\geq \frac{n}{2}\geq 2$.

\begin{lemma}\label{tf}
\begin{enumerate}
\item If $\Cliff(E)<\lfloor\frac{g-1}{2}\rfloor$, then the sheaf $N$ is torsion-free.
\item If $\Cliff(E)=\lfloor\frac{g-1}{2}\rfloor$ and $h^0(C,E)>4$, then $N$ is torsion-free unless $g=7$.
\end{enumerate}
\end{lemma}
\begin{proof}  Let $T\subset N$ be its torsion subsheaf, and $Q$ the quotient:
\[0\longrightarrow T\longrightarrow N\longrightarrow  Q\longrightarrow 0\]
Suppose $T$ is nonzero.  Then $Q$ is a rank $n-k$ torsion-free sheaf of degree $0$, with a surjection $\O_X^{\oplus n}\into Q$, so choosing a generic subsheaf $\O_X^{\oplus (n-k)}\subset \O_X^{\oplus n}$, we again have an isomorphism $Q\cong \O_X^{\oplus (n-k)}$.  This means that
\[v(T)=(0,H,k-\ell)\]
and further since $N\into Q$ must be surjective on global sections (since $\O_X^{\oplus n}\into N$ is an isomorphism on global sections), it follows that $h^0(X,T)=k$ and $h^1(X,T)=\ell$.  Since $M$ is torsion-free, clearly $T$ injects into $E$ and is therefore a purely $1$-dimensional rank 1 sheaf on $C$.  By the semistability of $E$ we then have
\[k-\ell\leq \frac{h^0(C,E)-h^1(C,E)}{2}\Rightarrow 0\leq k-\frac{h^0(C,E)}{2}\leq \ell-\frac{h^1(C,E)}{2}\]
using that $n=h^0(C,E)$.  It then follows that $\Cliff(T)\leq \Cliff(E)$, which contradicts the assumption in part (a).

\vskip 4pt

For part (b), first assume $g\neq 9$.  Then we would have $\Cliff(T)=\lfloor\frac{g-1}{2}\rfloor$, so by Proposition \ref{k3clifford} it must be the case that $k=h^0(C,T)=2$, whence $h^0(C,E)=4$ and we again have a contradiction.  The same proof works for $g=9$ as long as we now take the first column of \eqref{diagram} to be a Jordan-H\"older filtration of $F$.
\end{proof}

Now using Lemma \ref{bigrank}, we can conclude the proof of Mercat's Conjecture $(M_2)$ for curves lying on generic $K3$ surfaces.

\begin{proof}[Proof of Theorem \ref{k3mercat}]  We consider the sheaf $E$ on a curve $C\in |H|$ as above and assume $\Cliff(E)<\lfloor\frac{g-1}{2}\rfloor$.  Since the sheaves $G$ and $M$ defined via the diagram \eqref{diagram} are both stable, we must have the following inequalities:
\begin{align*}
v(G)^2\geq -2&\Rightarrow g\geq k\ell\\
v(M)^2\geq -2&\Rightarrow g\geq (n-k)(m-\ell)
\end{align*}
We know $k\geq 2$ and by Lemmas \ref{bigrank} and \ref{tf}(a), we have $n-k\geq 2$.  If $\ell\leq 1$, then we write
$$g\geq 2(m-\ell)\geq m+n-2$$ which is a contradiction; similarly, if $\ell\geq m-1$, we obtain that $g\geq k(m-1)\geq \frac{n}{2}(m-1)$, which is again a contradiction. Therefore $m-\ell\geq 2$.

\vskip 4pt

Note that the maximum value of $x+y$ for $(x,y)\in \mathbb Z^2$  satisfying the inequalities $g\geq xy$ and $x,y\geq 2$ is achieved for $(x,y)=\bigl(2,\lceil\frac{g-1}{2}\rceil\bigr)$,
with the exception of $(x,y)=(3,3)$ in the case $g=9$.  Therefore, since $n-k\geq 2$, we get
\[k+\ell \leq \Bigl\lceil\frac{g-1}{2}\Bigr \rceil+2\indent\mbox{ and }\indent (n+m)-(k+\ell) \leq \Bigl\lceil\frac{g-1}{2}\Bigr\rceil+2\]
which together imply $\frac{n+m}{2}\leq \lceil\frac{g-1}{2}\rceil+2$, contradicting the assumption that $\mbox{Cliff}(E)<\bigl \lfloor \frac{g-1}{2}\bigr \rfloor$.

\end{proof}

\vskip 5pt

We now show that on a generic curve of sufficiently high even genus, $\mbox{Cliff}_2(C)$ is only computed by vector bundles $E$ with $h^0(C,E)=4$.
For a partial result in this direction, see \cite{LN4}, Theorem 7.4. Note that for  a generic $K3$ section this result is not true. However, by specializing to curves on $K3$ surfaces,  we classify the rank $2$ vector bundles $E$ with $h^0(C,E)>4$ and $\mbox{Cliff}(E)=\mbox{Cliff}(C)$, then show that these bundles do not exist for a general curve $[C]\in \cM_g$.

\begin{prop}\label{equalclifford}  Let $E$ be a semistable rank $2$ torsion-free sheaf on $C\in|H|$ contributing to the Clifford index with $\Cliff(E)=\lfloor\frac{g-1}{2}\rfloor$.  If $h^0(C,E)>4$ and $g\neq 7, 9$, then $E\otimes K_C^\vee$ is the restriction of a Lazarsfeld--Mukai bundle of a rank $1$ torsion-free sheaf computing the Clifford index
$\mathrm{Cliff}(C)$. In particular, $\mathrm{det}(E)\cong K_C$.
\end{prop}
\begin{proof}  By Lemma \ref{tf}(b), $N$ is torsion-free and the computation of the previous proof still holds, showing that one of $k$ and $\ell$ is equal to $2$ and the other is $\lceil\frac{g-1}{2}\rceil$.  Likewise for $n-k$ and $m-\ell$.  Since $k\geq \frac{n}{2}>2$, we know $k=\lceil\frac{g-1}{2}\rceil$ and $\ell=2$.  We cannot have $m-\ell=2$ since $m\geq n> 4$, so $m-\ell=\lceil\frac{g-1}{2}\rceil$ and $n-k=2$.  We therefore have
\begin{align*}
v(M)&=\left(2,-H,\Bigl \lceil \frac{g-1}{2} \Bigr \rceil\right)\\
v(N)&=\left(2,H,\Bigl \lceil\frac{g-1}{2}\Bigr \rceil\right)
\end{align*}
The first thing to note is that $M$ has the same Mukai vector as the Lazarsfeld--Mukai bundle associated to a torsion-free sheaf $A$ computing the rank 1 Clifford index of $C$, and therefore it \emph{is} one.  Otherwise, a general plane in $H^2(X,M)$ would provide a section of $\O_X(H)$ with a line bundle $A$ of smaller Clifford index as the cokernel
\[0\longrightarrow M\longrightarrow \O^{\oplus 2}_X\longrightarrow A\longrightarrow  0.\]
Secondly, we have $v(M(H))=v(N)$, so by Lemma \ref{bigrank} we have that $N\cong M(H)$ and therefore $E\cong M(H)_{|C}\cong M_{|C}\otimes K_C$.
\end{proof}

\vskip 4pt

\begin{remark}  Let $\mathcal{P}_g$ be the stack of pairs $(X,C)$ where $X$ is a smooth $K3$ surface of genus $g$ and $C\subset X$ is a smooth curve in the polarization class.  Proposition \ref{equalclifford} implies a stronger form of a result of Arbarello, Bruno, and Sernesi \cite{ABS} carrying out Mukai's program in odd genus:  for $g=2s+1\geq 11$ and a general point $(X,C)\in\mathcal{P}_g$, the restriction map identifies the moduli space $M(v)$ of vector bundles $F$ on $X$ with $v(F)=v=(2,H,s)$ with the Brill--Noether stratum of semistable vector bundles $E$ on $C$ of degree $2g-2$ and $h^0(C,E)\geq s+2$.  Thus, we recover $X$ up to derived equivalence from a full Brill--Noether stratum on $C$.  It is shown in \cite{ABS} that $M(v)$ is a component of the same Brill--Noether stratum with fixed (canonical) determinant.
\end{remark}

\vskip 4pt

Proposition \ref{equalclifford} has several implications to the form of a rank 2 vector bundle $E$ on a general curve $C$ of genus $g$ computing $\mathrm{Cliff}_2(C)$.  If $h^0(C,E)>4$, then $\mu(E)=g-1$, and further, $h^0(C,E)=h^1(C,E)=2+\lceil\frac{g-1}{2}\rceil$.  As already pointed out, if $E$ has no subpencils, then
\[h^0(C,\det(E))\geq 2h^0(C,E)-3,\]
which is impossible if either $g$ is even, or $g$ is odd and $\det(E)\not\cong K_C$.

\begin{lemma}\label{even}  Let $E$ be a semistable rank $2$ vector bundle on a general curve $C$ with $h^0(C,E)>4$ and $\mathrm{Cliff}(E)=\bigl\lfloor \frac{g-1}{2}\bigr\rfloor$.  If  $g$ is even or $g$ is odd and $\det(E)\not\cong K_C$, then $E$ sits in a sequence
\[0\longrightarrow A\longrightarrow E\longrightarrow K_C(-A')\longrightarrow 0,\]
where both $A$ and $A'$ are line bundles computing $\mathrm{Cliff}(C)$.
\end{lemma}
\begin{proof} By the above, $E$ has a subpencil, so there is an exact sequence
\[0\longrightarrow A\longrightarrow E\longrightarrow K_C(-A')\longrightarrow 0\]
and by the semistability of $E$, the line bundle $A$ contributes to the Clifford index of $C$.  Assume first that $\Cliff(A)> \Cliff(E)$.  Since
\[\Cliff(E)\geq\frac{1}{2}\Bigl(\Cliff(A)+\Cliff(A')\Bigr),\]
it follows that $\Cliff(A')<\Cliff(E)$, so $A'$ cannot contribute to the Clifford index of $C$.   We have $\deg(A')\leq g-1$, so this is only the case if $h^0(C, A')<2$.  But then
\[h^1(C,A)\geq h^1(C,E)-h^1(C,K_C(-A'))\geq h^1(C,E)-1=1+\Bigl\lceil\frac{g-1}{2}\Bigr\rceil\]
which contradicts $\Cliff(A)>\Cliff(E)$.  Thus,
\[\Cliff(A)=\Cliff(A')=\Cliff(E)=\Cliff(C).\]
\end{proof}

We are in a position to complete the proof of Theorem \ref{4sect}.

\begin{proof}[Proof of Theorem \ref{4sect}]
We fix a general curve $C$ of genus $g=2a\geq 10$, in particular $C$ does not lie on a $K3$ surface. Using \cite{V1} Proposition 4.2, we may assume that the multiplication map $$\mbox{Sym}^2 H^0\bigl(C, K_C(-A)\bigr)\rightarrow H^0\bigl(C, K_C^{\otimes 2}(-2A)\bigr)$$ is surjective for every $A\in W^1_{a+1}(C)$. If $E$ is a rank $2$ vector bundle with $h^0(C,E)>4$ computing $\mbox{Cliff}_2(C)$, then from Lemma \ref{even}, necessarily $h^0(C,E)=a+2$ and $E$ sits in an exact sequence
\begin{equation}\label{ext3}
0\longrightarrow A\longrightarrow E\longrightarrow K_C(-A')\longrightarrow 0,
\end{equation}
for pencils $A, A'\in W^1_{a+1}(C)$. The existence of $E$ implies that the multiplication map
$$\mu_{A,A'}:H^0(C,K_C(-A))\otimes H^0(C,K_C(-A'))\rightarrow H^0(K_C^{\otimes 2}(-A-A'))$$
is \emph{not} surjective. Using \cite{V1} \emph{loc.cit.}, the pencils $A$ and $A'$ must be distinct, in particular, by the Base Point Free Pencil Trick, $h^0(C, A\otimes A')\geq 4$.

\vskip 3pt

We now let $C$ specialize to a generic $K3$ section. From Proposition \ref{equalclifford}, the corresponding line bundles $A$ and $A'$
such that $\mu_{A,A'}$ is not surjective must be \emph{equal}, for $\mbox{det}(E)=K_C$. On the other hand, since $A$ and $A'$ are obtained as specialization of \emph{distinct}
line bundles on neighboring curves, $h^0(C,A^{\otimes 2})\geq 4$. Equivalently, the Petri map $H^0(C,A)\otimes H^0(C, K_C(-A))\rightarrow H^0(C,K_C)$ is not injective.
This contradicts \cite{La} for $C$ satisfies the Petri Theorem, hence no such vector bundle $E$ on a general curve can exist.
\end{proof}

\begin{remark} We certainly expect a result similar to Theorem \ref{4sect} for a general curve $C$ of odd genus $g=2a+1\geq 15$. Applying \cite{T}, there can exist no semistable rank $2$ vector bundles $E$ on $C$ with $\mbox{det}(E)\cong K_C$ and $h^0(C,E)=a+2$, for $a>6$. Thus we may assume that $\mbox{det}(E)\not\cong K_C$. Via Proposition \ref{equalclifford}, it suffices to show that for each pair of pencils $A,A'\in W^1_{a+2}(C)$, the multiplication map
$$\mu_{A,A'}:H^0(C,K_C(-A))\otimes H^0(C,K_C(-A'))\rightarrow H^0(K_C^{\otimes 2}(-A-A'))$$
is surjective. This we expect to hold, but so far the surjectivity of the map $\mu_{A,A'}$ has been established only for a general pair $(A,A')\in W^1_{a+2}(C)\times W^1_{a+2}(C)$, see
\cite{V1}.
\end{remark}

\section{Koszul divisors on $\mm_g$ and the Mercat Conjecture}

In this section we prove Theorem \ref{divisor}. We fix an odd genus $g:=2a+1\geq 11$ and denote by $\mathcal{H}$ the Hurwitz space parametrizing pairs $[C,A]$, where $C$ is a smooth curve of genus $g$ and $A\in W^1_{g-4}(C)$ is a pencil. It is well-known that $\mathcal{H}$ is irreducible of dimension
$$4g-13=\mbox{dim}(\cM_g)+\rho(g,1,g-4).$$ Let $\sigma:\mathcal{H}\rightarrow \cM_g$ be the forgetful morphism. If $A\in W^1_{g-4}(C)$, we set $L:=K_C\otimes A^{\vee}\in W^4_{g+2}(C)$ and consider the multiplication map of global sections:
$$\mu_L:\mbox{Sym}^2 H^0(C,L)\rightarrow H^0(C,L^{\otimes 2}).$$
We introduce the following degeneracy locus in the Hurwitz space $$\cZ:=\Bigl\{[C,A]\in \mathcal{H}:K_{1,1}(C,K_C(-A))\neq 0\Bigr\}.$$ Since for a Brill-Noether general curve $C$, we have $H^1(C, L^{\otimes 2})=0$ for each $L\in W^4_{g+2}(C)$, the expected dimension of $\cZ$ is equal to
$$\mbox{dim}(\mathcal{H})-\Bigl(\chi(C,L^{\otimes 2})-\mbox{dim } \mbox{Sym}^2 H^0(C,L)+1 \Bigr)=\mbox{dim}(\cM_g)-1.$$
We set $\mathfrak{Kosz}_g:=\sigma_*(\cZ)$ and we expect $\mathfrak{Kosz}_g$ to be a divisor on $\cM_g$. We first confirm that $\cZ$ carries a component whose image under $\sigma$ has codimension $1$ in $\cM_g$.

\vskip 5pt

\begin{prop}\label{cod1}
The degeneracy locus $\cZ$ has at least one irreducible component $\cU$ of dimension $3g-4$, mapping generically finitely onto a divisor of $\cM_g$.
\end{prop}
\begin{proof}
Suppose $Q\subset \pp^4$ is a smooth quadric and denote by $\mbox{Hilb}_Q$ the Hilbert scheme of curves $C\subset Q$  of genus $g:=2a+1$ and degree $2a+3$. By induction on $a\geq 5$, we shall construct a smooth curve $[C\hookrightarrow Q]$ corresponding to a smooth point of $\mbox{Hilb}_Q$, such that the differential at the point $[C\hookrightarrow Q ]$ of the projection map $\pi:\mbox{Hilb}_Q\dashrightarrow \cM_g$ has rank equal to $3g-4$.
Passing to cohomology in the short exact sequence
$$0\longrightarrow T_C\longrightarrow T_{Q|C}\longrightarrow N_{C/Q}\longrightarrow 0,$$
these requirements will be satisfied once we construct $C\subset Q$ such that
\begin{equation}\label{conditions}
H^1(C,N_{C/Q})=0 \ \ \mbox{ and } \ \ h^1(C, T_{Q|C})=1,
\end{equation}
where $T_{Q|C}:=T_Q\otimes \mathcal{O}_C$. As usual, via Kodaira-Spencer theory, the differential $(d\pi)_{[C\hookrightarrow Q]}$ is identified with the coboundary map $\delta:H^0(C, N_{C/Q})\rightarrow H^1(C,T_C)$.
From the exact sequence
\begin{equation}\label{normquad}
0\longrightarrow N_{C/Q}\longrightarrow N_{C/\pp^4}\longrightarrow L^{\otimes 2}\longrightarrow 0,
\end{equation}
we compute $\mbox{deg}(N_{C/Q})=10a+9$, hence $\chi(N_{C/Q})=6a+9=3g-4+\mbox{dim } \mbox{Aut}(Q)$. If both conditions
(\ref{conditions}) are satisfied, the quotient space $\mbox{Hilb}_Q/\mbox{Aut}(Q)$ is smooth of dimension $3g-4$ at the point corresponding to $C\subset Q$ and the differential of the forgetful map to $\cM_g$ is injective.

\vskip 5pt

Assume, we have constructed $C\subset Q$ of genus $2a+1$ and degree $2a+3$ satisfying (\ref{conditions}). We pick general points $x,y,z\in C$, then set $\Lambda:=\langle x,y,z\rangle\ \subset \pp^4$ and $\gamma:=Q\cap \Lambda$. Thus $\gamma$ is a $3$-secant smooth conic to $C$ and we set $X:=C\cup \gamma$. The nodal curve $X\subset Q$ has arithmetic genus $p_a(X)=g(C)+2=2a+3$ and degree $2a+5$. Adapting the arguments from \cite{S} to the case of curves lying on a quadric, in order to show that $X$ is smoothable inside $Q$ to a curve $C'\subset Q$ satisfying $H^1(N_{C'/Q})=0$, it suffices to prove that
\begin{equation}\label{attach}
H^1(\gamma, N_{\gamma/Q})=0 \ \mbox{ and } \ H^1\bigl(\gamma, N_{\gamma/Q}(-x-y-z)\bigr)=0.
\end{equation}
Clearly $N_{\gamma/Q}=\mathcal{O}_{\gamma}(1)^{\oplus 2}$ and since $\mbox{deg}(\mathcal{O}_{\gamma}(1))=2$, both vanishings (\ref{attach}) are satisfied.

\vskip 3pt

In order to verify that $H^1(X, T_{Q|X})=0$, we write down the following exact sequence
$$ H^0(T_{Q|C})\oplus H^0(T_{Q|\gamma})\longrightarrow H^0(T_{Q|x+y+z})\longrightarrow H^1(T_{Q|X})\longrightarrow H^1(T_{Q|C})\oplus H^1(T_{Q|\gamma})\longrightarrow 0.$$
Since $T_{Q|\gamma}=\mathcal{O}_{\gamma}(1)^{\oplus 3}$, it follows that $H^1(T_{Q|\gamma})=0$ and the map $H^0(T_{Q|\gamma})\rightarrow H^0(T_{Q|x+y+z})$ is surjective. Therefore $H^1(X, T_{X|Q})\cong H^1(C, T_{Q|C})$, which establishes (\ref{conditions}) for the stable curve $X$ and completes the induction step. The base case $a=5$ follows from \cite{FO2} Section 5, for in genus $11$, the locus $\mathfrak{Kosz}_{11}$ has a concrete Noether-Lefschetz interpretation in terms of the (unique) $K3$ surface containing a general curve of genus $11$.
\end{proof}

\vskip 3pt

We now proceed and complete the proof of Theorem \ref{divisor}. To that end, we fix as before a $K3$ surface $(X,H)$ with $H^2=2g-2$ and $\mbox{Pic}(X)=\mathbb Z\cdot H$ and denote by $j:\pp^1\rightarrow \mm_g$ the rational curve induced by a Lefschetz pencil of curves in the linear system $|H|$.  The intersection numbers of $j(\pp^1)$ with the standard generators of $\mbox{Pic}(\mm_g)$ are well-known:
$$j^*(\lambda)=g+1, \ \ j^*(\delta_0)=6(g+3), \ \mbox{ and } \  j^*(\delta_i)=0, \ \mbox{ for } i=1, \ldots, \bigl\lfloor \frac{g}{2}\bigr\rfloor.$$
Note that all singular curves in the Leschetz pencil $j(\pp^1)$ are irreducible with a single node.

\begin{theorem}\label{kosz}
For each odd genus $g\geq 11$, the locus $\mathfrak{Kosz}_g$ is an effective divisor on $\cM_g$ different from the Hurwitz divisor. The slope of its closure in $\mm_g$ is equal to
$$s(\overline{\mathfrak{Kosz}}_g)=6+\frac{12}{g+1}.$$
\end{theorem}

\begin{proof}
Keeping the notation above, we claim that $j(\pp^1)\cap \overline{\mathfrak{Kosz}}_g=\emptyset$. Granting this for a moment, we obtain that $\mathfrak{Kosz}_g$ is a proper subvariety of $\cM_g$. Since from Proposition \ref{cod1}, the locus $\mathfrak{Kosz}_g$ has at least one divisorial component, necessarily $$s(\overline{\mathfrak{Kosz}}_g)=\frac{j^*(\delta_0)}{j^*(\lambda)}=6+\frac{12}{g+1}.$$

We establish the claim.  Assume there exists a curve $[C]\in j(\pp^1)$ and $L\in W^4_{2a+3}(C)$, such that $\mu_L$ is not injective. Then using (\ref{equiv1}), we construct a globally generated rank $2$ vector bundle $E$ on $C$ with $\mbox{det}(E)=L$ and $h^0(C,E)\geq 4$. We compute $\mbox{Cliff}(E)=a-\frac{1}{2}$. On the other hand, by \cite{La}, the curve $C$ satisfies the Brill-Noether Theorem, in particular $\mbox{Cliff}(C)=a$. We reach a contradiction with Theorem \ref{k3mercat}, by observing that $E$ must be stable. Indeed, assume that there exists a destabilizing exact sequence
$$0\longrightarrow B\longrightarrow E\longrightarrow L(-B)\longrightarrow 0.$$
On one hand, $\mbox{deg}(B)\geq a+2$, so $\mbox{deg}(L(-B))\leq a+1$. Furthermore, $E$ has no subpencils, hence $h^0(C,B)\leq 1$ and accordingly $h^0(C, L(-B))\geq 3$. On the other hand, $\rho(2a+1,2,a+1)<0$, which is a contradiction, for $C$ is Brill-Noether general.

Finally, in \cite{FO2} Theorem 1.1, using special $K3$ surfaces an example of a curve $[C]\in \mathfrak{Kosz}_g$ with $\mbox{Cliff}(C)=a$ is produced, hence $\mathfrak{Kosz}_g\nsubseteq \mathfrak{Hur}_g$, which finishes the proof.
\end{proof}



\subsection{The Hurwitz divisor viewed as a special Koszul divisor.}
Comparing the syzygetic description of the divisor $\mathfrak{Kosz}_g$ given in Theorem \ref{divisor}, it is natural to ask, whether the remaining component of the locus of curves $[C]\in \cM_g$ such that $\mbox{Cliff}_2(C)<\frac{g-1}{2}$ possesses a similar Koszul-theoretic incarnation. The answer is affirmative:

\begin{theorem}\label{bn}
For odd $g\geq 3$, the Hurwitz divisor $\mathfrak{Hur}_g$ is set-theoretically equal to the locus
$$\Bigl\{[C]\in \cM_g: \exists L\in W^2_{g+1}(C) \ \mbox{ } such \mbox{ }\  that \ \ K_{1,1}(C,L)\neq 0\Bigr \}.$$
\end{theorem}
\begin{proof}
Assume there exists $L\in W^2_{g+1}(C)$ such that $K_{1,1}(C,L)\neq 0$. We consider the induced rational map $\varphi_L:C\dashrightarrow \pp^2$.
The image $\varphi_L(C)$ lies on a conic, therefore there exists a line bundle $A\in \mbox{Pic}(C)$ with $h^0(C,A)\geq 2$, such that $H^0(C, L\otimes A^{\otimes (-2)})\neq 0$.
Since $C$ can be assumed to lie outside any codimension $2$ locus of $\cM_g$, clearly $C$ has gonality at least $\frac{g+1}{2}$, thus $L=A^{\otimes 2}$, and accordingly $[C]\in \mathfrak{Hur}_g$.

\vskip 4pt

Conversely, if $C$ is a general point of the (irreducible) divisor $\mathfrak{Hur}_g$, then $C$ has as a unique pencil $A\in W^1_{\frac{g+1}{2}}(C)$, with $h^0(C,A)=2$ and such that the multiplication map $$\mu:\mbox{Sym}^2 H^0(C,A)\rightarrow H^0(C,A^{\otimes 2})$$ is an isomorphism. We set $L:=A^{\otimes 2}\in W^2_{g+1}(C)$ and let $(s_1, s_2)$ be a basis of $H^0(C,A)$. Then $0\neq \bigl(\mu(s_1)\bigr)^2\cdot \bigl(\mu(s_2)\bigr)^2-\bigl(\mu(s_1\cdot s_2)\bigr)^2\in K_{1,1}(C,L)$.
\end{proof}

\vskip 4pt

\begin{remark} Theorem \ref{bn} can be regarded as a generalization of M. Noether's Theorem on the projective normality of non-hyperelliptic canonical curves. Indeed, since the only linear system of type $\mathfrak g^2_4$ on a curve of genus $3$ is its canonical bundle, Theorem \ref{bn}
asserts that a curve $[C]\in \cM_3$ is non-hyperelliptic if and only the multiplication map
$$\mbox{Sym}^2 H^0(C,K_C)\rightarrow H^0(C, K_C^{\otimes 2})$$ is an isomorphism.
\end{remark}

\end{document}